\documentclass[11pt,a4paper]{amsart}
\usepackage {mypackage}
\date{\today}

\newcolumntype{C}[1]{>{\centering\arraybackslash$}m{#1}<{$}}

\author{Nicolas Tholozan}
\address{University of Luxembourg, Campus Kirchberg \\
6, rue Richard Coudenhove-Kalergi\\
L-1359 Luxembourg}
\email{nicolas.tholozan@uni.lu}

\theoremstyle{theorem}
\newtheorem*{ExistenceProb}{Existence Problem}

\newtheorem*{VolumeProb}{Volume Problem}

\title[Volume of Clifford--Klein forms]{Volume and non-existence of compact Clifford--Klein forms}

\begin{document}

\maketitle

\begin{abstract}

This article studies the volume of compact quotients of reductive homogeneous spaces.
Let $G/H$ be a reductive homogeneous space and $\Gamma$ a discrete subgroup of $G$ acting properly discontinuously and cocompactly on $G/H$. We prove that the volume of $\Gamma \backslash G/H$ is the integral, over a certain homology class of $\Gamma$, of a $G$-invariant form on $G/K$ (where $K$ is a maximal compact subgroup of $G$).

As a corollary, we obtain a large class of homogeneous spaces the compact quotients of which have rational volume. For instance, compact quotients of pseudo-Riemannian spaces of constant curvature $-1$ and odd dimension have rational volume. This contrasts with the Riemannian case.

We also derive a new obstruction to the existence of compact Clifford--Klein forms for certain homogeneous spaces. In particular, we obtain that $\SO(p,q+1)/SO(p,q)$ does not admit compact quotients when $p$ is odd, and that $\SL(n,\R)/\SL(m,\R)$ does not admit compact quotients when $m$ is even.
\end{abstract}

\section*{Introduction}

The problem of understanding compact quotients of homogeneous spaces has a long history that can be traced back to the ``Erlangen program'' of Felix Klein \cite{Klein1872}. In the second half of the last century, the breakthroughs of Borel \cite{Borel63}, Mostow \cite{Mostow68}, Margulis \cite{Margulis91} and many others lead to a rather good understanding of quotients of \emph{Riemannian} homogeneous spaces. Comparatively, little is known about the non-Riemannian case, and in particular about quotients of pseudo-Riemannian homogeneous spaces.

In this paper we will mainly focus on reductive homogeneous spaces, i.e. quotients of a semi-simple Lie group $G$ by a closed reductive subgroup $H$. The $G$-homogeneous space $X=G/H$ carries a natural $G$-invariant pseudo-Riemannian metric (induced by the Killing metric of $G$) and therefore (up to taking a covering of degree $2$) a $G$-invariant volume form $\vol_X$. A quotient of $X$ by a discrete subgroup $\Gamma$ of $G$ acting properly discontinuously and cocompactly is called a \emph{compact Clifford--Klein form} of $X$, or (when it does not lead to any confusion) a \emph{compact quotient} of $X$.

 The study of compact reductive Clifford--Klein forms was initiated in the 80's by Kulkarni \cite{Kulkarni81} and Kobayashi \cite{Kobayashi89}. A lot of things remain to be understood, despite the significant works of Benoist \cite{Benoist96}, Kobayashi \cite{Kobayashi89,Kobayashi92,Kobayashi93,Kobayashi96,Kobayashi98}, Labourie \cite{BenoistLabourie92}, Mozes and Zimmer \cite{LMZ95}, Margulis \cite{Margulis97}, and more recently the works of Kassel \cite{Kassel08,Kassel10}, Gu\'eritaud \cite{GueritaudKassel}, Guichard and Wienhard \cite{GGKW}.\\

In this paper we will address the following two questions, to which no general answer is known:
 
\begin{ExistenceProb}
Which reductive homogeneous spaces admit compact Clifford--Klein forms?
\end{ExistenceProb}

\begin{VolumeProb}
Let $G/H$ be a reductive homogeneous space and $\Gamma$ a discrete subgroup of $G$ acting properly discontinuously and cocompactly on~$G/H$. Is the volume of $\Gamma \backslash G/H$ rational (up to a scaling constant independent of~$\Gamma$)?
\end{VolumeProb}

A particularly interesting family of homogeneous spaces are the \emph{pseudo-Riemannian homogeneous spaces of constant curvature}, a unified definition of which was given by Wolf in \cite{Wolf62}.
Recall that the pseudo-Riemannian homogeneous space of signature $(p,q)$ and constant negative curvature is the space
\[\H^{p,q} = \SO_0(p,q+1)/\SO_0(p,q)~.\] 
In this setting our results are summarized in the following:

\begin{MonThm}
Let $p$ and $q$ be positive integers. Then:
\begin{itemize}
\item If $p$ is odd, then $\H^{p,q}$ does not admit any compact Clifford--Klein form.
\item If $p$ is even, then the volume of any compact Clifford--Klein form of $\H^{p,q}$ is a rational multiple of the volume of the sphere of dimension $p+q$. 
\end{itemize}
\end{MonThm}

Prior to this work, the first point was only known when both $p$ and $q$ are odd \cite{Kulkarni81}, as well as when $p \leq q$ \cite{Wolf62}. The second point follows from the Chern--Gauss--Bonnet formula when $p+q$ is even but is new when $p$ is even and $q$ is odd.\\

Let us now give a more detailed overview of the results contained in this paper.

\subsection*{Volume of Compact Clifford--Klein forms}

It is well-known that the volume of a closed hyperbolic manifold of dimension $2n$ is essentially an integer, due to the Chern--Gauss--Bonnet formula. This argument generalizes to compact quotients of a reductive homogeneous space $G/H$ whenever one can show that the volume is a \emph{Chern--Weil class} associated to the canonical principal $H$-bundle over $G/H$ (see Section \ref{s:Rigidity}). If $G/H$ is a symmetric space, this is known to happen if and only if $G$ and $H$ have the same complex rank.

This argument has no chance to work for homogeneous spaces of odd dimension (because Chern--Weil classes have even degree), nor for homogeneous spaces of the form $H\! \times \! H/\Delta(H)$ (where $\Delta(H)$ denotes the diagonal embedding of $H$), for which all the Chern--Weil invariants are trivial. It is known for instance that the volume of a closed hyperbolic $3$-manifold is usually not rational.

In contrast, we proved in a recent paper (see \cite{Tholozan5}) that the volume of a closed \emph{anti-de Sitter} $3$-manifold (i.e. a compact quotient of $\H^{2,1}$) is a rational multiple of $\frac{\pi^2}{2}$, answering a question that was raised in \cite{QuestionsAdS}. The anti-de Sitter space $\H^{2,1}$ can be seen as the \emph{group space} $\SO_0(2,1)$ (i.e. the Lie group $\SO_0(2,1)$ with the action of $\SO_0(2,1) \times \SO_0(2,1)$ by left and right multiplication, see Definition~\ref{d:GroupSpace}). Its compact Clifford--Klein are known to exist and to have a rich deformation space (see \cite{Salein00}, \cite{KasselThese} or \cite{Tholozan3}). Kulkarni and Raymond proved in~\cite{KulkarniRaymond85} that these compact Clifford--Klein forms have the form
\[j\!\times\!\rho(\Gamma) \backslash \SO_0(2,1)~,\]
where $\Gamma$ is a cocompact lattice in $\SO_0(2,1)$, $j$ the inclusion and $\rho$ another representation of $\Gamma$ into $\SO_0(2,1)$. Moreover, Gu\'eritaud and Kassel proved in \cite{GueritaudKassel} that these quotients have the structure of a $\SO(2)$-bundle over $\Gamma \backslash \H^2$ (see Theorem \ref{t:FibrationGK}). In \cite{Tholozan5}, we proved the following formula:
\begin{equation} \label{eq:VolSO(n,1)} \Vol \left( j\!\times\!\rho(\Gamma) \backslash \SO_0(2,1)\right) = \frac{\pi^2}{2}\left(\euler(j) + \euler(\rho)\right)~,\end{equation}
where $\euler$ denotes the Euler class. This formula was later recovered by Alessandrini--Li \cite{AlessandriniLi15} and Labourie \cite{LabouriePrivate} using different methods.

It may seem surprising that a ``Chern--Weil-like'' invariant such as the Euler class appears when computing the volume of a $3$-manifold. The first aim of this paper is to explain better this phenomenon an generalize it to a much broader setting. 

The main issue is that we don't have a structure theorem similar to the one of Gu\'eritaud--Kassel in general (see Theorem \ref{t:FibrationGK} and the conjecture that follows). We will overcome this problem with the following argument: denoting $L$ a maximal compact subgroup of $H$ and $K$ a maximal compact subgroup of $G$ containing $L$, we see that $\Gamma\backslash G/H$ is homotopically equivalent to $\Gamma \backslash G/L$, which is a $K/L$-bundle over $\Gamma \backslash G/K$. Let $q$ be the dimension of $K/L$ and $p+q$ the dimension of $G/H$. A classical use of spectral sequences shows that $\Gamma$ has homological dimension $p$ and that $\HH_p(\Gamma, \Z)$ is generated by an element $[\Gamma]$ (Proposition \ref{p:HomDimGamma}). Since $G/K$ is contractible, $\HH_p(\Gamma, \Z)$ is naturally isomorphic to $\HH_p(\Gamma\backslash G/K, \Z)$ and $[\Gamma]$ can thus be realized as a singular $p$-cycle in $\Gamma\backslash G/K$. We will prove the following:

\begin{MonThm} \label{t:VolumeCliffordKlein}
Let $G/H$ be a reductive homogeneous space, with $G$ and $H$ connected and of finite center. Let $L$ be a maximal compact subgroup of $H$ and $K$ a maximal compact subgroup of $G$ containing $L$. Set $p = \dim G/H - \dim K/L$. Then there exists a $G$-invariant $p$-form $\omega_{G,H}$ on $G/K$ such that, for any torsion-free discrete subgroup $\Gamma \subset G$ acting properly discontinuously and cocompactly on $G/H$, we have
\[\Vol\left( \Gamma \backslash G/H\right) = \left| \int_{[\Gamma]} \omega_{G,H} \right|~.\]
\end{MonThm}

It turns out that, in many cases, the form $\omega_{G,H}$ is a ``Chern--Weil form'', though the volume form of $G/H$ is not (see Section \ref{s:InclusionSymSpaces}). This implies that the volume of any compact quotient of $G/H$ is a rational multiple of the volume of $G_U/H_U$, where $G_U$ and $H_U$ respectively denote the compact Lie groups dual to $G$ and $H$ (see Section \ref{s:Rigidity}). In particular, we will obtain the following:

\begin{MonThm} \label{t:RationalityVolume}
For the following pairs $(G,H)$, the volume of compact quotients of $G/H$ is a rational multiple of the volume of $G_U/H_U$:
\begin{itemize}
\item[$(1)$] $G = \SO(p,q+1)$, $H = \SO(p,q)$, $p$ even, $q>0$.
\item[$(2)$] $G = \SL(2n,\R)$, $H = \SL(2n-1,\R)$, $n > 0$.
\item[$(3)$] $G$ a Hermitian Lie group, $H$ any semi-simple subgroup.
\end{itemize}
\end{MonThm}

Cases $(1)$ and $(2)$ concern families of symmetric spaces that have attracted a lot of interest. However, they potentially carry no information. Indeed the symmetric space $\SL(2n,\R)/\SL(2n-1,\R)$ is conjectured not to admit any compact quotient (see next subsection), and the only known compact quotients of $\H^{p,q} = \SO(p,q+1)/\SO(p,q)$ for $p\geq 3$ are the so-called \emph{standard} quotients constructed by Kulkarni in \cite{Kulkarni81}, for which the theorem reduces to a classical statement about volumes of quotients of Riemannian symmetric spaces. Non standard quotients are only known in the case of $\H^{2,1}$, which was treated in \cite{Tholozan5} (see Equation \eqref{eq:VolSO(n,1)}) and \cite{AlessandriniLi15}.

Case $(3)$, on the other side, shows in particular that the volume of a compact quotient of the group space $\SU(d,1)$ is a rational multiple of the volume of $\SU(d+1)$. These compact quotients are known to exist and some of them have rich deformation spaces, as was proven by Kobayashi \cite{}, Kassel \cite{Kassel10}, and Gu\'eritaud--Guichard--Kassel--Wienhard \cite{GGKW}. Like quotients of $\SO_0(2,1)$, they are known to have (up to a finite cover) the form
\[j\!\times\!\rho(\Gamma) \backslash \SU(d,1)~,\]
where $\Gamma$ is a uniform lattice in $\SU(d,1)$, $j: \Gamma \to \SU(d,1)$ is the inclusion and $\rho:\Gamma \to \SU(d,1)$ is another representation (see Theorem \ref{t:QuotientsSU(d,1)} for a more precise statement). For such Clifford--Klein forms, we will actually give a more precise formula. Recall that $\SU(d,1)$ acts transitively on the complex hyperbolic space $\H^d_\C$ and preserves a K\"ahler form $\omega$. If $\Gamma$ is a uniform lattice in $\SU(d,1)$ and $\rho:\Gamma \to \SU(d,1)$ a representation, we define
\[\tau_k(\rho) = \int_{\Gamma \backslash \H^d_\C} \omega^{d-k} \wedge f^*\omega^k~,\]
where $f: \H^d_\C \to \H^d_\C$ is any smooth $\rho$-equivariant map.

\begin{MonThm} \label{t:VolumeQuotientsSU(d,1)Intro}
Let $\Gamma$ be a lattice in $\SU(d,1)$, $j:\Gamma \to \SU(d,1)$ the inclusion and $\rho: \Gamma \to \SU(d,1)$ another representation such that $j\! \times\! \rho(\Gamma)$ acts properly discontinuously and cocompactly on $\SU(d,1)$. Then
\[\Vol\left( j\!\times \! \rho(\Gamma) \backslash \SU(d+1)\right) = \Vol(\SU(d+1)) \sum_{k=0}^d \tau_k(\rho)~.\]
\end{MonThm}




\subsection*{A new obstruction to the existence of compact quotients}

Contrary to the Riemannian setting, compact pseudo-Riemannian Clifford--Klein forms do not always exist, and it is a long standing problem to characterize which reductive homogeneous spaces admit compact quotients. This question lead to many important works of Kulkarni \cite{Kulkarni81}, Kobayashi \cite{Kobayashi89,Kobayashi92,Kobayashi96}, Benoist \cite{Benoist96}, Labourie, \cite{BenoistLabourie92}, Mozes, Zimmer \cite{LMZ95,LabourieZimmer95}, Margulis \cite{Margulis97} or Shalom \cite{Shalom00}. We refer to \cite{KobayashiYoshino05} or \cite{Constantine12} for a more thorough survey. Let us recall here two famous conjectures that emerged from these works. 

\begin{KobaConj}
The homogeneous space $\H^{p,q} = \SO_0(p,q+1)/\SO_0(p,q)$ ($p,q>0$) admits a compact Clifford--Klein form if and only if one of the following holds:
\begin{itemize}
\item $p$ is even and $q=1$,
\item $p$ is a multiple of $4$ and $q=3$,
\item $p=8$ and $q=7$.
\end{itemize}
\end{KobaConj}

\begin{conj}[See for instance \cite{KasselThese}, Section 0.1.5]
The homogeneous space $\SL(n,\R)/\SL(m,\R)$ ($1<m<n$) never admits a compact Clifford--Klein form.
\end{conj}

In this paper, we obtain a powerful cohomological obstruction, allowing us to do significant advances toward these conjectures. In Section \ref{s:NonExistence}, we prove that in many cases the form $\omega_{G,H}$ of Theorem \ref{t:VolumeCliffordKlein} vanishes, directly implying that the reductive homogeneous space $G/H$ does not admit a compact Clifford--Klein form. In particular, we obtain the following:

\begin{MonThm}\label{t:AdvanceKobayashiConj}
For the following pairs $(G,H)$, the homogeneous space $G/H$ does not have any compact Clifford--Klein form.
\begin{itemize}
\item[$(1)$] $G= \SO_0(p,q+r)$, $H= \SO_0(p,q)$, $p,q,r>0$, $p$ odd;
\item[$(2)$] $G= \SL(n,\R)$, $H= \SL(m,\R)$, $1<m<n$, $m$ even;
\item[$(4)$] $G = \SL(p+q,\C)$, $H = \SU(p,q)$, $p,q>0$;
\item[$(5)$] $G = \Sp(2(p+q),\C)$, $H = \Sp(p,q)$;
\item[$(6)$] $G = \SO(2n,\C)$, $H = \SO^*(2n)$;
\item[$(7)$] $G = \SL(p+q,\R)$, $H = \SO_0(p,q)$, $p,q>1$;
\item[$(8)$] $G = \SL(p+q, \mathcal{H})$, $H = \Sp(p,q)$, $p,q>1$.
(Here $\mathcal{H}$ denotes the field of quaternions.)
\end{itemize}
\end{MonThm}

All of these cases are partly new. They were obtained independently by Morita in \cite{Morita16}. We give more details about how these results relate to earlier works in Section \ref{ss:EarlierResults} and to Yosuke Morita's work in Section \ref{ss:Morita}.

Finally, our obstruction will allow us to prove the following theorem, which was conjectured by Kobayashi (see \cite[Conjecture 4.15]{Kobayashi96}):

\begin{MonThm} \label{t:KobayashiRankConj}
Let $G$ be a connected semi-simple Lie group, $H$ a connected semi-simple subgroup of $G$, $L$ a maximal compact subgroup of $H$ and $K$ a maximal compact subgroup of $G$ containing $L$. If
\[\rk(G) - \rk(K) < \rk(H) - \rk(L)~\]
(where $\rk$ denotes the complex rank), then $G/H$ does not have a compact Clifford--Klein form.
\end{MonThm}
Note that Morita \cite{MoritaPreprint} independently proved that this theorem is implied by a previous result of his \cite{Morita15}.

\subsection*{Organization of the paper}

In Section \ref{s:HomologicalFibration}, we explain why compact reductive Clifford--Klein forms behave like fibrations over an Eilenberg--MacLane space ``at the homology level''. In Section \ref{s:FiberwiseIntegration} we construct the form $\omega_{G,H}$ as the contraction of a $(p+q)$-form on $G/L$ along the fibers $gK/L$ and we prove Theorem \ref{t:VolumeCliffordKlein}. In Section \ref{s:CompactDual}, we study the form corresponding to $\omega_{G,H}$ on the compact dual symmetric space $G_U/K$ and show that this form is ``Poincar\'e-dual'' to the inclusion of $H_U/L$ in $G_U/K$. In Section \ref{s:InclusionSymSpaces}, we derive a condition under which the form $\omega_{G,H}$ vanishes and a condition under which it is a ``Chern--Weil'' class. In Section \ref{s:Rigidity}, we explain why, when $\omega_{G,H}$ is a Chern--Weil class, the volume of compact Clifford--Klein forms is rational, concluding the proof of Theorem \ref{t:RationalityVolume}. In Section \ref{s:GroupSpaces}, we describe the form $\omega_{G,H}$ in the case of group spaces and deduce Theorem \ref{t:VolumeQuotientsSU(d,1)Intro}. In Section \ref{s:NonExistence} we give three different ways of proving the vanishing of the form $\omega_{G,H}$, leading to Theorems \ref{t:AdvanceKobayashiConj} and \ref{t:KobayashiRankConj}. Finally in Section \ref{s:LocalFibrations}, we prove that the vanishing of the form $\omega_{G,H}$ is also an obstruction to the existence of certain local foliations of $G/H$ by compact homogeneous subspaces, and we formulate a conjecture about the geometry of compact reductive Clifford--Klein forms.

\subsection*{Acknowledgements}
I am very thankful to Gabriele Mondello and Gregory Ginot for helping me understand spectral sequences, to Bertrand Deroin for suggesting the use of Thom's representation theorem in the proof of Theorem~\ref{t:VolumeCliffordKlein}, to Yosuke Morita for many insightful discussions about our respective works, to Toshiyuki Kobayashi for remarks on a previous version of this article, and to Yves Benoist for encouraging me to improve this previous version.




\section{Clifford--Klein forms are fibrations at the homology level} \label{s:HomologicalFibration}

In all this paper, $G$ will denote a connected Lie group and $H$ a closed connected subgroup of $G$. We will also fix $L$ a maximal compact subgroup of $H$ and $K$ a maximal compact subgroup of $G$ containing $L$. According to the Cartan--Iwasawa--Malcev theorem, $L$ and $K$ are well-defined up to conjugation. We denote respectively by $\g$, $\h$, $\k$ and $\l$ the Lie algebras of $G$, $H$, $K$ and $L$.

We will assume that the action of $G$ on the homogeneous space $X=G/H$ preserves a volume form.
Recall that this is equivalent to requiring that
\[\Det(G)_{|H} = \Det(H)~,\]
where $\Det(G)$ and $\Det(H)$ denote respectively the \emph{modular functions} of $G$ and $H$. Starting from Section \ref{s:CompactDual}, we will assume $G$ and $H$ to be reductive and therefore unimodular, in which case this condition is automatically satisfied.

A \emph{compact Clifford--Klein form} of $X$ is a quotient of $X$ by a discrete subgroup $\Gamma$ of $G$ acting properly discontinuously and cocompactly. The $G$-invariant volume form $\vol_X$ then descends to a volume form on $\Gamma \backslash X$ (that we still denote by $\vol_X$) and we can define the \emph{volume} of $\Gamma \backslash X$ by
\[\Vol\left( \Gamma \backslash X \right) = \left | \int_{\Gamma \backslash X} \vol_X \right|~. \\ \]

Recall that, since $K$ and $L$ are maximal compact subgroups of $G$ and $H$ respectively, the homogeneous spaces $G/K$ and $H/L$ are contractible. Let us fix a torsion-free discrete subgroup $\Gamma$ of $G$ acting properly discontinuously and cocompactly on $G/H$, and denote by $M$ the Clifford--Klein form
\[M= \Gamma \backslash G/H~.\]
We introduce two auxiliary Clifford--Klein forms: 
\[E = \Gamma \backslash G/L\] and
\[B = \Gamma \backslash G/K~.\]
($E$ and $B$ are smooth manifolds since $\Gamma$ is discrete and torsion-free.)

We remark the following facts:
\begin{itemize}
\item[$(i)$] $E$ fibers over $M$ with fibers isomorphic to $H/L$. Since $H/L$ is contractible, this fibration is a homotopy equivalence.
\item[$(ii)$] $E$ also fibers over $B$ with fibers isomorphic to $K/L$.
\item[$(iii)$] Since $G/K$ is contractible, $B$ is a \emph{classifying space} for $\Gamma$.
\end{itemize}

From the first point, we deduce in particular that the homology of $M$ is the same as the homology of $E$. The third point implies that the homology of $B$ is the homology of $\Gamma$. Finally, $(ii)$ implies that the homologies of $B$, $E$ and $K/L$ are linked (in an elaborate way) by the \emph{Leray--Serre spectral sequence}. We will use the following classical consequence:

\begin{prop}[See \cite{Kulkarni81} and \cite{Kobayashi89}] \label{p:HomDimGamma}
Let $q$ denote the dimension of $K/L$ and $p+q$ the dimension of $G/H$. Then the group $\Gamma$ has homological dimension $p$ and
\[\HH_p(\Gamma, \Z) \simeq \HH_{p+q}(M,\Z) \simeq \Z~.\]
\end{prop}

\begin{proof}
Let $p'$, $q'$ and $r'$ denote respectively the homological dimensions of $B$, $K/L$ and $E$. By Serre's theorem, the spectral sequence given by 
\[\E_{k,l}^2 = \HH_k \left(B, \HH_l(K/L,\Z)\right)\]
converges to $\HH_{k+l}(E,\Z)$. A classical consequence is that
\[r' = p' + q'\]
and that 
\begin{equation} \label{eq:SpectralSequence}
\HH_{p'+q'}(E,\Z) \simeq \HH_{p'}\left(B, \HH_{q'}(K/L, \Z) \right)~.
\end{equation}
Since $K/L$ is a closed oriented manifold of dimension $q$, we have $q'=q$ and $\HH_q(K/L,\Z) \simeq \Z$. Since $E$ is homotopy equivalent to $M$ which is a closed oriented manifold of dimension $p+q$, we also have $r'= p+q$. Therefore $p'= p$.

Moreover, since $L$ is connected, the action of $\Gamma$ on $G/L$ preserves an orientation of the fibers of the fibration 
\[G/L \to G/K\]
and $\Gamma$ thus acts trivially on $\HH_q(K/L, \Z)$. From \eqref{eq:SpectralSequence}, we obtain
\[\Z \simeq \HH_{p+q}(E,\Z) \simeq \HH_p(B,\Z)~.\]
The proposition follows since $E$ is homotopy equivalent to $M$ and $B$ is a classifying space for $\Gamma$.
\end{proof}

To go further, we need to explicitly describe the isomorphism $\HH_{p+q}(E,\Z) \simeq \HH_p(B,\Z)$. Let $[\Gamma]$ denote a generator of $\HH_p(B,\Z) \simeq \HH_p(\Gamma,\Z)$, and $\pi$ the fibration of $E$ over $B$. Roughly speaking, if one thinks of $[\Gamma]$ as a closed submanifold of $B$ of dimension $p$, then the isomorphism $\HH_p(B,\Z) \overset{~}{\to} \HH_{p+q}(E,\Z)$ maps $[\Gamma]$ to $\pi^{-1}([\Gamma])$, which is a submanifold of $E$ of dimension $p+q$.

However, we don't know whether $[\Gamma]$ can be represented by a submanifold. One way to overcome this difficulty would be to work with simplicial complexes. However, since we will use differential geometry later, it is more convenient to use Thom's realization theorem:

\begin{CiteThm}[Thom, \cite{Thom54}]
There exists a closed oriented $p$-manifold $B'$, a smooth map $\phi:B' \to B$ and an integer $k$ such that 
\[ k [\Gamma] = \phi_*[B']~,\]
where $[B']$ denotes the fundamental class of $B'$.
\end{CiteThm}

Let $\pi': E'\to B'$ be the pull-back of the fibration $\pi:E \to B$ by $\phi$ and $\hat{\phi}: E' \to E$ the lift of $\phi$. The total space of the fibration $E'$ is a closed orientable $(p+q)$-manifold.

\begin{prop}
Let $[E]$ denote a generator of $\HH_{p+q}(E)$ and $[E']$ denote the fundamental class of $E'$. Then, up to switching the orientation of $E'$, we have
\[k [E] = \hat{\phi}_*[E']~.\]
\end{prop}

\begin{proof}
The Leray--Serre spectral sequence shows that the fibrations $\pi$ and $\pi'$ respectively induce isomorphisms
\[\pi^*: \HH_p(B) \to \HH_{p+q}(E)\]
and 
\[{\pi'}^*: \HH_p(B') \to \HH_{p+q}(E')~.\]
By naturality of the Serre spectral sequence, we have the following commuting diagram:
\[
\xymatrix{
\HH_p(B') \ar[d]_{{\pi'}^*} \ar[r]^{\phi_*} & \HH_p(B) \ar[d]^{\pi^*} \\
\HH_{p+q}(E') \ar[r]^{\hat{\phi}_*} & \HH_{p+q}(E)~.
}
\]
Now, $B'$ and $E'$ are closed oriented manifolds of dimension $p$ and $p+q$ respectively. Since ${\pi'}^*$ is an isomorphism, it maps the fundamental class of $B'$ to the fundamental class of $E'$ (up to switching the orientation of $E'$). Since $\phi_*[B'] = k [\Gamma]$, we thus have
\[\hat{\phi}_*[E'] = k [E]~.\]
\end{proof}

To summarize, we proved that the rational homology of $E$ in dimension $p+q$ is generated by a cycle that ``fibers'' over a $p$-cycle of $B$.

\section{Fiberwise integration of the volume form} \label{s:FiberwiseIntegration}

Let $E'$, $B'$, $\phi$, $\hat{\phi}$ and $\pi$, $\pi'$ be as in the previous section. Denote by $\psi$ the projection from $E$ to $M$. Recall that the volume form $\vol_X$ on $X=G/H$ induces a volume form on $M$ that we still denote by $\vol_X$.

Since $\psi$ is a homotopy equivalence, we have
\[\Vol(M) = \left| \int_M \vol_X \right| = \left| \int_{[E]} \psi^*\vol_X \right|~.\]
Since $k[E] = \hat{\phi}_*[E']$, we have
\[\left| \int_{[E]} \psi^*\vol_X \right| = \frac{1}{k}\left| \int_{E'} \hat{\phi}^*\psi^*\vol_X \right|~.\]

Now, since $E'$ fibers over $B'$, we can ``average'' the form $\hat{\phi}^*\psi^*\vol_X$ along the fibers to obtain a $p$-form on $B'$ whose integral will give the volume of~$M$. Let $x$ be a point in $G/K$ and let $F$ denote the fiber $\pi^{-1}(x)$. Choose some volume form $\vol_F$ on $F$ and let $\xi$ denote the section of $\Lambda^q TF$ such that $\vol_F(\xi)=1$. At every point $y$ of $F$, the $p$-form obtained by contracting $\psi^*\vol_X$ with $\xi$ has $T_y F$ in its kernel and therefore induces a $p$-form $\omega_y$ on $T_x G/K$. 

\begin{defi} \label{d:OmegaH} The form $\omega_{G,H}$ on $G/K$ is defined at the point $x$ by
\[(\omega_{G,H})_x = \int_F \omega_y\ \d \vol_F(y)~.\]
\end{defi}
One easily checks that this definition does not depend on the choice of $\vol_F$. Since the maps $\psi$ and $\pi$ are equivariant with respect to the actions of~$G$, the volume forms $\psi^* \vol_X$ and $\omega_{G,H}$ are $G$-invariant. By a slight abuse of notation, we still denote by $\omega_{G,H}$ the induced $p$-form on $B = \Gamma \backslash G/K$.

\begin{prop} \label{p:FiberIntegration}
For any submanifold $V$ of dimension $p$ in $G/K$, we have 
\begin{equation} \int_V \omega_{G,H} = \int_{\pi^{-1}(V)} \psi^*\vol_X~. \end{equation} 
\end{prop}

\begin{proof}
This is presumably a classical result of differential geometry. Let $U$ be an open subset of $V$ over which the fibration $\pi$ is trivial. Let us identify $\pi^{-1}(U)$ with $K/L \times U$. We can locally write the form $\psi^*\vol_X$ as $f(y,x) \vol_F \wedge \vol_U$ for some function $f$ on $K/L\times U$ and some volume forms $\vol_F$ and $\vol_U$ on $K/L$ and $U$ respectively. Let $\xi$ be the section of $\Lambda^q T K/L$ such that $\vol_F(\xi) = 1$. The contraction of $\psi^*\vol_X$ with $\xi$ is thus $f(y,x) \vol_U$. By construction, we thus have
\[(\omega_{G,H})_x = \left(\int_F f(x,y) \d \vol_F(y)\right) \vol_U~,\]
and therefore
\begin{eqnarray*}
\int_{\pi^{-1}(U)} \psi^*\vol_X & = & \int_{F\times U} f(y,x) \d \vol_F(y) \d \vol_U(x) \\
\ & = & \int_U \omega_{G,H}~.
\end{eqnarray*}
\end{proof}

In particular, if $V$ is a sphere of dimension $p$ in $G/K$ that can be homotoped to a point $p$, then $\pi^{-1}(V)$ can be homotoped to the fiber $\pi^{-1}(p)$. We thus have
\[\int_V \omega_{G,H} = \int_{\pi^{-1}(V)} \psi^*\vol_X = 0~.\]
Since $\psi^*\vol_X$ is closed. This shows that $\omega_{G,H}$ is closed.

\begin{rmk}
In the following, we will assume that $G$ is semi-simple, in which case any $G$-invariant form on $G/K$ is closed, according to a well-known theorem of Cartan.
\end{rmk}

We can now conclude the proof of Theorem \ref{t:VolumeCliffordKlein}. Indeed, we have
\begin{eqnarray*}
\Vol(M) & = & \frac{1}{k} \left| \int_{E'} \hat{\phi}^* \psi^* \vol_X \right|\\
\ & = & \frac{1}{k} \left| \int_{B'} \phi^* \omega_{G,H} \right| \quad \textrm{by Proposition \eqref{p:FiberIntegration}}\\
\ & = & \left| \int_{[\Gamma]} \omega_{G,H} \right|~.\\
\end{eqnarray*}

Let us conclude this section by giving a more explicit way to compute the form $\omega_{G,H}$ when $G$ is a connected semi-simple Lie group with finite center. Recall that in that case, the tangent space of $G/K$ at the point $x_0 = K$ can be identified with the orthogonal of $\h$ in $\g$ with respect to the Killing form of $\g$. Moreover, the form $\omega_{G,H}$ is uniquely determined by its restriction to $T_{x_0}G/K$.

If $\frak{v}$ is a subspace of $\g$ of dimension $d$ in restriction to which the Killing form $\Kill_G$ is non degenerate, we denote by $\omega_{\frak{v}}$ the $d$-form on $\g$ given by composing the orthogonal projection on $\frak{v}$ with the volume form on $\frak{v}$ induced by the restriction of the Killing form.

Finally, let us provide $K/L$ with the left invariant volume form $\omega_{K/L}$ induced by the restriction of the metric on $G/H$.
\begin{lem} \label{l:ComputationOmegaGH}
The form $\omega_{G,H}$ at the point $x_0$ is given by
\[(\omega_{G,H})_{x_0} = \int_{K/L} \Ad_u^* \omega_{\k^\perp \cap \h^\perp}\ \d \omega_{K/L}(u)~.\]
\end{lem}

\begin{proof}
In the construction of $\omega_{G,H}$ (Definition \ref{d:OmegaH}), we choose $\omega_{K/L}$ as our volume form on $F_{x_0} = K/L$. Let $\xi$ be the $q$-vector on $\omega_{K/L}$ such that $\omega_{K/L}(\xi) = 1$.

At $y_0 = L$, the pull-back of $\vol_X$ by the projection $\psi : G/L \to G/H$ identifies with the form $\omega_{\h^\perp}$ on $\g$. Since the $q$-vector $\xi$ at $y_0$ is given by $e_1\wedge \ldots \wedge e_q$, where $(e_1,\ldots , e_q)$ is an orthonormal frame of $\k \cap \h^\perp$, we have
\[(i_\xi\omega_{\h^\perp})_{y_0} = \omega_{\k^\perp \cap \h^\perp}~.\]
By left invariance, we also have
\[(i_\xi\psi^*\vol_X)_{u\cdot y_0} = u_*\omega_{\k^\perp \cap \h^\perp}~.\]

Now, identifying $T_{u\cdot y_0}G/L$ with $u_* \l^\perp$, the differential of $\pi: G/L \to G/K$ is given at $u\cdot y_0$ by
\begin{eqnarray*}
\d \pi_{u\cdot y_0}(u_* v) & = & \dt \pi(u \exp(t v) \cdot y_0) \\
 & = & \dt \pi \left( \exp(t \Ad_u(v)) u \cdot y_0 \right) \\
 & = & \dt \exp(t \Ad_u(v)) \cdot \pi(u \cdot y_0)\\
 & = & \dt \exp(t \Ad_u(v)) \cdot x_0\\
 & = & p_{\k^\perp}\Ad_u(v)~,
\end{eqnarray*}
where $p_{\k^\perp}$ denotes the orthogonal projection on $\k^\perp$.

Therefore, the form $(i_\xi\psi^*\vol_X)$ at $u\cdot y_0$, whose kernel contains $u_* \k$, induces by projection the form ${\Ad_u}_* \omega_{\k^\perp \cap \h^\perp}$ at $x_0$. By construction of the form $\omega_{G,H}$, we thus obtain
\[(\omega_{G,H})_{x_0} = \int_{K/L} {\Ad_u}_* \omega_{\k^\perp \cap \h^\perp}\ \d \omega_{K/L}(u)~.\]
\end{proof}

\section{The corresponding form on the compact dual} \label{s:CompactDual}

From now on, we assume that $G$ is a connected semi-simple Lie group with finite center and that $H$ is a reductive subgroup. In this section we investigate the form $\omega_{G,H}^U$ corresponding to $\omega_{G,H}$ on the \emph{compact dual} of $G/K$.\\

Write
\[\g = \k \oplus \p~,\]
where $\p$ is the orthogonal of $\k$ with respect to the Killing form.
Then $\k \oplus i \p$ is a Lie subalgebra of the complexification $\g^\C$ of $\g$, generating a compact Lie group $G_U$ containing $K$, called the \emph{compact dual} of $G$. The compact symmetric space $G_U/K$ is the \emph{compact dual} of the symmetric space $G/K$.

By construction, the tangent spaces at the base point $x_0 = K$ in $G/K$ and $G_U/K$ are isomorphic as representations of $K$. This induces an isomorphism between the exterior algebras of invariant forms on $G/K$ and $G_U/K$. If $\alpha$ is a $G$-invariant form on $G/K$, the image of $\alpha$ by this isomorphism will be called the \emph{form corresponding to $\alpha$ on the compact dual} and will be denoted $\alpha^U$.

The group $G_U$ contains the compact dual $H_U$ of $H$, and one can define a map $\iota: H_U/L \to G_U/K$. This map may not be injective, but it is a covering of finite degree onto its image, since $L$ is a finite index subgroup of $H_U\cap K$. We denote by $[H_U/L]$ the fundamental class of $H_U/L$. 

\begin{defi}
Let $N$ be a closed oriented manifold of dimension $d$ and $[c]$ a rational homology class of degree $k$ on $N$. Let \[\vee: \HH_k(N,\Q) \times \HH_{n-k}(N,\Q)\to \Q\] denote the intersection pairing. The cohomology class $[\alpha] \in \HH^{d-k}(N,\Q)$ is called \emph{Poincaré-dual} to $[c]$ if for any $[c'] \in \HH_{d-k}(N,\Q)$, one has
\[\int_{[c']} [\alpha] = [c] \vee [c']~.\] 
\end{defi}
According to Poincar\'e's duality theorem, every rational homology class of a closed oriented manifold has a unique Poincar\'e-dual cohomology class.

\begin{theo} \label{t:PoincareDual}
The cohomology class of the form 
\[\frac{1}{\Vol(G_U/H_U)}\omega_{G,H}^U \in \HH^\bullet(G_U/K,\Q)\] 
is Poincar\'e-dual to the homology class $\iota_*[H_U/L]$.
\end{theo}

\begin{proof}
Let $x_0$ denote the point $K$ in $G_U/K$. By Lemma \ref{l:ComputationOmegaGH}, we have
\[(\omega_{G,H}^U)_{x_0} = \int_{K/L} \Ad_u^* \omega_{\k^\perp \cap i\h^\perp}\ \d \omega_{K/L}(u)~.\]

Thus, if $\phi$ denotes the projection from $G_U/L$ to $G_U/H_U$ and $\pi$ the projection from $G_U/L$ to $G_U/K$, then one can reproduce word by word the arguments of the previous section and show that
\[\int_C \frac{1}{\Vol(G_U/H_U)}\omega_{G,H}^U = \int_{\pi^{-1}(C)} \frac{1}{\Vol(G_U/H_U)} \phi^*\vol_{G_U/H_U}\]
for any oriented submanifold $C$ of $G_U/K$ of dimension $p$.

Now, the form $\frac{1}{\Vol(G_U/H_U)}\vol_{G_U/H_U}$ is Poincar\'e-dual to the homology class of a point in $G_U/H_U$, and $\phi^*\vol_{G_U/H_U}$ is thus dual to the homology class of the fiber $H_U/L \subset G_U/L$ of the map $\phi$.

Therefore, $\int_{\pi^{-1}(C)} \frac{1}{\Vol(G_U/H_U)} \phi^*\vol_{G_U/H_U}$ counts the homological intersection number between $H_U/L$ and $\pi^{-1}(C)$ in $G_U/L$. This is equal to $k$ times the homological intersection number between $\iota(H_U/L)$ and $C$ in $G_U/K$, where $k$ denotes the degree of the covering map $\iota : H_U/L \to H_U/H_U\!\cap\! K$. Hence $\int_{\pi^{-1}(C)} \frac{1}{\Vol(G_U/H_U)} \phi^*\vol_{G_U/H_U}$ is equal to $[C] \vee \iota_*[H_U/L]$. The conclusion follows.
\end{proof}

\section{Cohomology and inclusion of symmetric spaces} \label{s:InclusionSymSpaces}

In this section, we go deeper into the cohomology theory of symmetric spaces in order to find conditions under which the form $\omega_{G,H}^U$ vanishes and conditions under which it is a \emph{Chern--Weil form}.

We say that $\omega_{G,H}^U$ is a \emph{Chern--Weil form} if its cohomology class is a Chern--Weil characteristic class of the canonical principal $K$-bundle over $G_U/K$ (see Section \ref{s:Rigidity} for details). Our aim is to prove the following theorem:

\begin{theo} \label{t:RankCondition}
Let $\rk$ denote the complex rank of a Lie group.
\begin{itemize}
\item The form $\omega_{G,H}^U$ vanishes when
\[\rk(H_U)- \rk(L) > \rk(G_U) - \rk(K)~.\]
\item If $\omega_{G,H}^U$ does not vanish, then it is a Chern--Weil form if and only if
\[\rk(H_U)- \rk(L) = \rk(G_U) - \rk(K)~.\]
\end{itemize}
\end{theo}

The cohomology of symmetric spaces has been described by the works of Cartan and Borel in the years 1950 \cite{Cartan50,Borel53}. This description is summarized in the following theorem:
\begin{CiteThm}[Cartan] \label{t:CohomologySymSpace}
Let $G_U/K$ be a symmetric space of compact type, with $K$ connected. Then
\begin{itemize}
\item The cohomology algebra $\HH^\bullet(G_U/K,\Q)$ is isomorphic to a tensor product
\[\HH_{even}^\bullet(G_U/K,\Q) \otimes \HH_{odd}^\bullet(G_U/K,\Q)~,\]
\item the subalgebra $\HH_{even}^\bullet(G_U/K,\Q)$ is the algebra of Chern--Weil classes of the canonical principal $K$-bundle over $G_U/K$, and is concentrated in even degree,
\item the subalgebra $\HH_{odd}^\bullet(G_U/K,\Q)$ is isomorphic to $\Lambda^\bullet\left(\Prim(G_U/K,\Q)\right)$ where $\Prim(G_U/K,\Q)$ is a vector subspace of dimension $\rk(G_U) - \rk(K)$ generated by elements of odd degree,
\end{itemize}
\end{CiteThm}

The cohomology algebra of a symmetric space thus has the structure of a bi-graded algebra:
\[\HH^\bullet(G_U/K,\Q) = \bigoplus_{p,q \geq 0} \HH_{even}^p(G_U/K,\Q) \otimes \HH_{odd}^q(G_U/K,\Q)~.\]
We will say that a cohomology class $\alpha$ has bi-degree $(p,q)$ if it belongs to $\HH_{even}^p(G_U/K,\Q) \otimes \HH_{odd}^q(G_U/K,\Q)$.

\begin{prop} \label{p:PreserveBigrading}
The map $\iota^*: \HH^\bullet(G_U/K,\Q) \to \HH^\bullet(H_U/L,\Q)$ maps $\HH_{even}^p(G_U/K,\Q)$ to $\HH_{even}^p(H_U/L,\Q)$ and $\HH_{odd}^p(G_U/K,\Q)$ to $\HH_{odd}^p(H_U/L,\Q)$, and thus preserves the bi-grading. Moreover, it maps $\Prim(G_U/K,\Q)$ to $\Prim(H_U/L,\Q)$.
\end{prop}

This proposition is likely to be a straightforward consequence of the proof of Cartan's theorem. We prove it in the forthcoming paper \cite{Tholozan8}.\\

If $G_U/K$ is a symmetric space of compact type, let us denote by $\dimeven(G_U/K)$ and $\dimodd(G_U/K)$ the maximal degree of a non zero cohomology class in $\HH^\bullet_{even}(G_U/K,\Q)$ and $\HH^\bullet_{odd}(G_U/K,\Q)$, respectively. Since $G_U/K$ is compact and orientable, we obtain by Cartan's theorem that
\[\dimeven(G_U/K) + \dimodd(G_U/K) = \dim(G_U/K)\]
and that
\[\HH^{\dimeven(G_U/K)}_{even}(G_U/K,\Q) \otimes \HH^{\dimodd(G_U/K)}_{odd}(G_U/K,\Q) = \HH^{\dim(G_U/K)}(G_U/K,\Q)~.\]
Thus, both $\HH^{\dimeven(G_U/K)}_{even}(G_U/K,\Q)$ and $\HH^{\dimodd(G_U/K)}_{odd}(G_U/K,\Q)$ have dimension~$1$.

\begin{prop} \label{p:InjectivityOdd}
If $\iota_*[H_U/L]$ does not vanish in $\HH_\bullet (G_U/K, \Q)$, then the homomorphism
\[\iota^*: \HH^{\dimeven(H_U/L)}_{even}(G_U/K,\Q) \to \HH^{\dimeven(H_U/L)}_{even}(H_U/L,\Q)\]
is surjective, and the morphism
\[\iota^*: \Prim(G_U/K,\Q) \to \Prim(H_U/L,\Q)\]
is surjective.
\end{prop}

\begin{proof}
If $\iota_*[H_U/L]$ does not vanish in $\HH_\bullet (G_U/K, \Q)$, then, by Poincaré duality, there exists an element $\alpha \in \HH^{\dim(H_U/L)}(G_U/K,\Q)$ such that $\iota^*\alpha \neq 0$. By Cartan's theorem, we can write
\[\alpha = \sum_{k+l = \dim(H_U/L)} \beta_k \otimes \gamma_l~,\]
with $\beta_k \in \HH^k_{even}(G_U/K,\Q)$ and $\gamma_l \in \HH^{l}_{odd}(G_U/K,\Q)$.

Since $\iota^*\beta_k = 0$ for $k> \dimeven(H_U/L)$ and $\iota^*\gamma_l = 0$ for $l> \dimodd(H_U/L)$, we get that
\[\iota^*\alpha = \iota^*\beta_{\dimeven(H_U/L)} \otimes \iota^*\gamma_{\dimodd(H_U/L)}\neq 0~,\]
which implies that both $i^*\beta_{\dimeven(H_U/L)}$ and $\iota^*\gamma_{\dimodd(H_U/L)}$ do not vanish. Since $\HH^{\dimeven(H_U/L)}(H_U/L,\Q)$ and $\HH^{\dimodd(H_U/L)}(H_U/L,\Q)$ are one dimensional, we conclude that
\[\iota^*: \HH^{\dimeven(H_U/L)}_{even}(G_U/K,\Q) \to \HH^{\dimeven(H_U/L)}_{even}(H_U/L,\Q)\]
and 
\[\iota^*: \HH^{\dimodd(H_U/L)}_{odd}(G_U/K,\Q) \to \HH^{\dimodd(H_U/L)}_{odd}(H_U/L,\Q)\]
are surjective.

Now, by Cartan's theorem, $\HH^\bullet_{odd}(H_U/L, \Q) = \Lambda^\bullet \Prim(H_U/L,\Q)$. If $\iota^*:\Prim(G_U/K,\Q)\to \Prim(H_U/L,\Q)$ were not surjective, then $\iota^*\left(\HH^\bullet_{odd}(G_U/K,\Q)\right)$ would be included in $\Lambda^\bullet F$ for a proper subspace $F$ of $\Prim(H_U/L,\Q)$, and it would not contain any form of top degree. Since
\[\iota^*: \HH^{\dimodd(H_U/L)}_{odd}(G_U/K,\Q) \to \HH^{\dimodd(H_U/L)}_{odd}(H_U/L,\Q)\]
is surjective, we conclude that $\iota^*:\Prim(G_U/K,\Q)\to \Prim(H_U/L,\Q)$ is surjective.
\end{proof}

We can now prove Theorem \ref{t:RankCondition}.

\begin{proof}[Proof of Theorem \ref{t:RankCondition}]
Assume that $\omega_{G,H}^U$ does not vanish. Then, by Proposition \ref{t:PoincareDual}, $\iota_*[H_U/L]$ does not vanish in $\HH_\bullet(G_U/K,\Q)$. By Proposition \ref{p:InjectivityOdd}, the map $\iota^*: \Prim(G_U/K,\Q) \to \Prim(H_U/L,\Q)$ is surjective, which implies that
\[\rk(H_U)-\rk(L) = \dim \Prim(H_U/L,\Q) \leq \dim \Prim(G_U/K,\Q) = \rk(G_U) - \rk(K)~.\]
This proves the first point. \\

Now, since $\omega_{G,H}^U$ is Poincaré dual to $i_*[H_U/L]$, we have
\[\int_{H_U/L} i^*\alpha = \int_{G_U/K} \alpha \wedge \omega_{G,H}^U\]
for all $\alpha \in \HH^{\dim(H_U/L)}(G_U/K,\Q)$.
In particular, for all $(k,l)$ such that $k+l=\dim(H_U/L)$ and for all  $\alpha \in \HH^{k+l}(G_U/K,\Q)$ of bi-degree $(k,l)$, we have $\int_{G_U/K} \alpha \wedge \omega_{G,H}^U = 0$ unless
\[(k,l) = \left(\dimeven(H_U/L), \dimodd(H_U/L)\right)~.\]
This implies that $\omega_{G,H}^U$ has bi-degree \[\left(\dimeven(G_U/K) - \dimeven(H_U/L), \dimodd(G_U/K) - \dimodd(H_U/L)\right)~.\]
Therefore, $[\omega_{G,H}^U]$ belongs to $\HH^\bullet_{even}(G_U/K,\Q)$ if and only if
\begin{equation} \label{eq:EqualityOddDim} \dimodd(G_U/K) = \dimodd(H_U/L)~.\end{equation}
Since $\iota^*: \Prim(G_U/K,\Q) \to \Prim(H_U/L, \Q)$ is surjective, Equality \eqref{eq:EqualityOddDim} happens if and only if it is also injective, which is equivalent to
\[\rk(H_U)-\rk(L) = \rk(G_U) - \rk(K)~.\]
This concludes the proof of Theorem \ref{t:RankCondition}.
\end{proof}

\section{Characteristic classes and rationality of the volume} \label{s:Rigidity}

In this section, we explain why, when $\omega_{G,H}^U$ is a Chern--Weil form, the volume of every compact quotient of $G/H$ is a rational multiple of $\Vol(G_U/H_U)$. This is a classical argument which relies on the fact that, by Proposition \ref{t:PoincareDual}, the form $\frac{1}{\Vol(G_U/H_U)}\omega_{G,H}^U$ represents an \emph{integral} cohomology class.

The precise result that we will prove is the following:
\begin{theo} \label{t:RationalVolumePrecise}
Assume that we have the equality:
\[\rk(G_U) - \rk(K) = \rk(H_U) - \rk(L)~.\]
Then there exists an integer $d$ such that, for any torsion-free discrete subgroup of $G$ acting properly discontinuously on $G/H$, the volume $\Vol(\Gamma \backslash G/H)$ is an integral multiple of $\frac{1}{d}\Vol(G_U/H_U)$.
\end{theo}

\begin{rmk}
Note that, given a normalization of the volume form on $G/H$, there is a canonical way to normalize the volume on $G_U/H_U$ accordingly. Thus the statement of Theorem \ref{t:RationalVolumePrecise} does not depend on the choice of such a normalization.
\end{rmk}

\begin{proof}[Proof of Theorem \ref{t:RationalVolumePrecise}]

Let $BK$ be a classifying space for $K$ and $EK \to BK$ be the associated universal principal $K$-bundle. There exists a map $f: G_U/K \to BK$, unique up to homotopy, such that the principal $K$-bundle $G_U$ is isomorphic to $f^*EK$. The map $f$ induces a homomorphism
\[f^*: \HH^\bullet(BK,\R) \to \HH^\bullet (G_U/K, \R)~.\]
By Theorem \ref{t:CohomologySymSpace} and by definition of Chern--Weil classes, the image of $f^*$ is the subalgebra $\HH^\bullet_{even}(G_U/K,\R)$. It contains as a lattice the $\Z$-module $f^* \HH^\bullet(BK,\Z)$.

It follows from Proposition \ref{t:PoincareDual} that the form $\frac{1}{\Vol(G_U/H_U)}\omega_{G,H}^U$ represents an integral cohomology class. 
Moreover, we saw in the previous section that, under the condition $\rk(G_U) - \rk(K) = \rk(H_U) - \rk(L)$, this cohomology class belongs to $\HH^\bullet_{even}(G_U/K,\R)$. Therefore, the cohomology class $\frac{1}{\Vol(G_U/H_U)}[\omega_{G,H}^U]$ belongs to the $\Z$-module $\Lambda = \HH^\bullet_{even}(G_U/K,\R) \cap \HH^\bullet(G_U/K,\Z)$. Since we have
\[f^* \HH^\bullet(BK,\Z) \subset \Lambda\]
and since $f^* \HH^\bullet(BK,\Z)$ is a lattice in $\HH^\bullet_{even}(G_U/K,\R)$, we obtain that $f^* \HH^\bullet(BK,\Z)$ has finite index in $\Lambda$. Therefore, there exists an integer $d$ such that \[\frac{d}{\Vol(G_U/H_U)}[\omega_{G,H}^U] \in f^* \HH^\bullet(BK,\Z)~.\]

Let us now denote by $\Sym^\bullet(\k)^K$ the algebra of polynomials on $\k$ invariant by the adjoint action of $K$. The Chern--Weil theory gives the existence of an isomorphism
\[\Phi: \HH^\bullet(BK,\R) \to \Sym^\bullet(\k)^K\]
such that, for any smooth map $f$ from a manifold $M$ to $BK$ and for any cohomology class $\alpha$ in $\HH^\bullet(BK,\R)$, the class $f^*\alpha$ in $\HH^\bullet(M,\R)$ is represented by the differential form $\Phi(\alpha)(F_\nabla)$, where $F_\nabla$ is the curvature of any connection on the principal bundle $f^*EK$. We denote by $\Sym_\Z^\bullet(\k)^K$ the image by $\Phi$ of $\HH^\bullet(BK,\Z)$.

Let $\nabla$ and $\nabla^U$ denote respectively the connections on the $K$-principal bundles over $G/K$ and $G_U/K$ given by the distribution orthogonal to the fibers (with respect to the Killing metric). These connections (hence their curvature forms) are respectively $G$ and $G_U$-invariant.

By the preceeding remarks, there is a polynomial $P \in \Sym_\Z^\bullet(\k)^K$ such that 
$\frac{d}{\Vol(G_U/H_U)}\omega_{G,H}^U$ and $P(F_{\nabla^U})$ are cohomologous. Since both forms are $G_U$-invariant, we actually have
\[\frac{d}{\Vol(G_U/H_U)}\omega_{G,H}^U=P(F_{\nabla^U})~.\]
By duality between the symmetric spaces $G_U/K$ and $G/K$, we then have
\[\frac{d}{\Vol(G_U/H_U)}\omega_{G,H} = (-1)^{\deg P}\, P(F_\nabla)~.\]
Let us denote by~$\alpha$ the inverse image of $P$ by the Chern--Weil isomorphism $\Phi$.

By Theorem \ref{t:VolumeCliffordKlein}, we have
\begin{eqnarray*}
d\,\frac{\Vol(\Gamma \backslash G/H)}{\Vol(G_U/H_U)} & = & \left | \int_{[\Gamma]} \frac{d}{\Vol(G_U/H_U)}\omega_{G,H} \right|\\
 & = & \left|\int_{[\Gamma]} P(F_\nabla)\right| \\
 & = & \left| \int_{[\Gamma]} f^*\alpha\right|~,
\end{eqnarray*}
where $f: \Gamma\backslash G/K \to BK$ is such that the $K$-principal bundle $\Gamma \backslash G$ over $\Gamma\backslash G/K$ is isomorphic to $f^*EK$. Since $\alpha$ belongs to $\HH^\bullet(BK,\Z)$, we obtain that $\frac{d \Vol(\Gamma \backslash G/H)}{\Vol(G_U/H_U)}$ is an integer. This proves Theorem \ref{t:RationalVolumePrecise}.
\end{proof}

Finally, let us conclude the proof of Theorem \ref{t:RationalityVolume}. Recall that the complex rank of $\SO(n)$ is $\left \lfloor \frac{n}{2}\right \rfloor$ and that the complex rank $\SL(n,\R)$ is $n-1$. It is then a simple computation to verify that the equality $\rk(H_U) - \rk(L) = \rk(G_U) - \rk(K)$ is satisfied in cases $(1)$ and $(2)$. For case $(3)$, it is a well-known fact that $\rk(G_U)=\rk(K)$ when $G_U/K$ is Hermitian (see \cite[Proposition 2.3]{HartnickOtt12}). In that case, any $G_U$-invariant form is a Chern--Weil form. In particular, $\omega_{G,H}^U$ is a Chern--Weil form (which vanishes if $\rk(H_U) - \rk(L) >0$).

\section{The case of group manifolds} \label{s:GroupSpaces}

In this section, we specify the previous results in the case of compact quotients of \emph{group spaces}.

\begin{defi} \label{d:GroupSpace}
A group space is a semi-simple Lie group $H$ provided with the action of $H\times H$ given by
\[(g,h)\cdot x = gxh^{-1}\]
for all $(g,h)\in H\times H$ and all $x\in H$.
\end{defi}
The group space $H$ can also be presented as the quotient $H\times H/\Delta(H)$, where $\Delta(H)$ denotes the diagonal embedding of $H$ in $H\times H$.

Group spaces form a large class of pseudo-Riemannian symmetric spaces (the pseudo-Riemannian metric being the Killing metric on $H$) which is interesting to study for several reasons.

First, given a compact Clifford--Klein form $\Gamma \backslash G/H$ of a reductive homogeneous space and a uniform lattice $\Lambda$ in $H$, one can construct the double quotient
\[\Gamma \backslash G/\Lambda~,\]
which is a compact Clifford--Klein form of the group space $G$. In order to understand all compact Clifford--Klein forms of reductive homogeneous spaces, it is thus enough (in theory) to understand compact quotients of group spaces.

The second motivation for studying group spaces is that, when $H$ has rank one, its compact Clifford--Klein forms are well-understood, thanks to results of Kobayashi \cite{Kobayashi93,Kobayashi98}, Kassel \cite{Kassel08}, Gu\'eritaud \cite{GueritaudKassel}, Guichard and Wienhard \cite{GGKW}.

Let $\Gamma$ be a uniform lattice in $H$ and $\rho: \Gamma \to H$ a homomorphism. We denote by $\Gamma_\rho$ the graph of $\rho$, i.e. the subgroup of $H\times H$ defined by
\[\Gamma_\rho = \{(\gamma, \rho(\gamma)) , \gamma \in \Gamma\}~.\]
The \emph{translation length} of an element $h\in H$ is defined by
\[l(h) = \inf_{x\in H/L} d(x,h\cdot x)~,\]
where $d$ is the distance associated to the $H$-invariant symmetric Riemannian metric on $H/L$.
We say that the homomorphism $\rho$ is \emph{uniformly contracting} if there exists $\lambda < 1$ such that for any $\gamma \in \Gamma$,
\[ l(\rho(\gamma)) \leq \lambda l(\gamma)~.\]

\begin{theo}[Kobayashi \cite{Kobayashi93}, Kassel \cite{Kassel08}, Gu\'eritaud--Guichard--Kassel--Wienhard \cite{GGKW}] \label{t:QuotientsSU(d,1)}
Let $H$ be a Lie group of rank $1$. Then every torsion-free discrete subgroup of $H\times H$ acting properly discontinuously and cocompactly on $H$ is equal to $\Gamma_\rho$ for some uniform lattice $\Gamma$ in $H$ an some contracting homomorphism $\rho:\Gamma \to H$.
\end{theo}
Conversely, Benoist--Kobayashi's properness criterion \cite{Benoist96,Kobayashi96} implies that such a group $\Gamma_\rho$ does act properly discontinuously and cocompactly on $H$.\\

The purpose of this section is to express the volume of $\Gamma_\rho \backslash H$ when $H = \SO_0(d,1)$ or $\SU(d,1)$ in terms of classical invariants associated to the representation $\rho$.\footnote{The case where $H$ is another Lie group of rank $1$ (namely $\Sp(d,1)$ of $\mathrm{F}_4)$ is not interesting because the representation $\rho$ must be virtually trivial, according to the super-rigidity theorem of Corlette \cite{Corlette92}.} In the case of~$\SO_0(d,1)$, we will recover the main theorem of \cite{Tholozan5}.\\

In order to do so, we first give a general way to compute the form $\omega_{G,H}$ for any group space $H\times H/\Delta(H)$, knowing the algebra of $H$-invariant forms on $H/L$. We thus restrict to the case where $G = H\times H$ acts on $X = H$ by left and right multiplication. To simplify notations, we denote by $\omega_H$ the form $\omega_{H\times H, \Delta(H)}$ constructed in Section \ref{s:FiberwiseIntegration} and by $\omega_H^U$ the corresponding form on the compact dual. The forms $\omega_H$ and $\omega_H^U$ are respectively a $H\times H$-invariant form on $H/L \times H/L$ and a $H_U \times H_U$-invariant form on $H_U/L \times H_U/L$.

Let $X$ be a compact oriented manifold of dimension $d$. We denote by $\vee$ the homological intersection pairing of $X$ and by $\wedge$ the cohomological product. For $0\leq k \leq d$, let us fix a basis $(e_1^k, \ldots , e_{n_k}^k)$ of the torsion-free part of $\HH_k(X, \Z)$. Let us denote by $({e_1^k}^*, \ldots , {e_{n_k}^k}^*)$ the dual basis for the intersection pairing, i.e. the basis of the torsion-free part of $\HH_{d-k}(X,\Z)$ characterized by
\[e_i^k \vee e_j^{d-k} = \delta_{ij}~.\]
Finally, let us denote by $(\alpha_1^k,\ldots , \alpha_{n_k}^k)$ and $({\alpha_1^k}^*, \ldots , {\alpha_{n_k}^k}^*)$ the bases of $\HH^k(X,\Q)$ and $\HH^{d-k}(X,\Q)$ satisfying respectively
\[\int_{e_i^k} \alpha_j^k = \delta_{ij}\]
and 
\[\int_{{e_i^k}^*} {\alpha_j^k}^* = \delta_{ij}~.\]

Recall that the cohomology ring of $X\times X$ is naturally isomorphic to the tensor product 
\[\HH^\bullet(X,\Q) \otimes \HH^\bullet(X,\Q)~.\]

\begin{defi}
We call \emph{Lefschetz cohomology class} on $X \times X$ the cohomology class of degree $d$ defined by
\[\beta_{Lef} = \sum_{k=0}^d (-1)^{d-k} \sum_{i = 1}^{n_k} \alpha_i^k \otimes {\alpha_i^k}^*~.\]
\end{defi}

The Lefschetz cohomology class on $H_U/L \times H_U/L$ can be represented by a unique $H_U\times H_U$-invariant form that we call the \emph{Lefschetz form}. We also call \emph{Lefschetz form} the corresponding $H\times H$-invariant form on the dual symmetric space $H/L\times H/L$.\\

The following proposition characterizes the Lefschetz cohomology class and shows in particular that it does not depend on our choice of basis for the homology.
\begin{prop}\label{p:DiagonalDual}
The Lefschetz cohomology class of $X$ is Poincar\'e-dual to the diagonal embedding of $X$ in $X\times X$. 
\end{prop}

In particular, when integrating the Lefschetz cohomology class on the graph of some map $f:X \to X$, one recovers the Lefschetz trace formula. Hence our choice of terminology.\\

\begin{proof}
Let $\Delta_X$ denote the diagonal embedding of $X$ in $X\times X$. We want to prove that for any $u\in \HH_d(X\times X, \Q)$, the number $\int_u \beta_{Lef}$ equals the homological intersection number between $u$ and $\Delta_X$. Since
\[\HH_d(X,\Q) = \bigoplus_{k=0}^d \HH_k(X,\Q) \otimes \HH_{d-k}(X,\Q)~,\]
it is enough to prove it for $u$ of the form $e_i^k \otimes {e_j^k}^*$, for all $0\leq k\leq d$ and all $1\leq i,j\leq n_k$.

By definition of $\beta_{Lef}$, we have
\[\int_{e_i^k \otimes {e_j^k}^*} \beta_{Lef} = (-1)^{d-k} \delta_{ij}~.\]

On the other side, intersections between (cycles representing) $e_i^k \otimes {e_j^k}^*$ and $\Delta_X$ correspond exactly to intersections between $e_i^k$ and ${e_j^k}^*$. Indeed, $e_i^k$ intersects ${e_j^k}^*$ at a point $x\in X$ if and only if $e_i^k \times e_j^{n-k}$ intersects $\Delta_X$ at $(x,x)$. Taking orientations into account, one checks that a positive intersection between $e_i^k$ and ${e_j^k}^*$ gives an intersection of sign $(-1)^{d-k}$ between $e_i^k \otimes {e_j^k}^*$ and $\Delta_X$. We thus obtain
\[\left( e_i^k \otimes {e_j^k}^*\right) \vee \Delta_X = (-1)^{d-k} e_i^k \vee {e_j^k}^*= (-1)^{d-k} \delta_{ij}~.\]
\end{proof}

By Proposition \ref{t:PoincareDual}, the form $\frac{1}{\Vol(H_U)}\omega_H^U$ on $G_U/K = H_U/L\times H_U/L$ is Poincaré dual to the diagonal embedding of $H_U/L$. By Propostion \ref{p:DiagonalDual}, we thus get:

\begin{coro}
The form $\frac{1}{\Vol(H_U)} \omega_H$ is the Lefschetz form on $H/L\times H/L$.\\
\end{coro}

Let us now apply this corollary to the case where $H$ is $\SO_0(d,1)$ or $\SU(n,1)$.

Let $\vol_{\H^d}$ denote the volume form on the hyperbolic space $\H^d$, which is the symmetric space of $\SO_0(d,1)$. If $\Gamma$ is a uniform lattice in $\SO_0(d,1)$ and $\rho: \Gamma \to \SO_0(d,1)$ a homomorphism, we define the \emph{volume} of $\rho$ by
\[\Vol(\rho) = \int_{\H^d/\Gamma} f^*\vol_{\H^d}~,\]
where $f: \H^d \to \H^d$ is any $\rho$-equivariant map.

Let $\omega$ denote the K\"ahler form on the complex hyperbolic space $\H^d_\C$, which is the symmetric space of $\SU(d,1)$. We normalize $\omega$ so that the corresponding form on the compact dual symmetric space $\ProjC{d}$ is a generator of $\HH^2(\ProjC{d},\Z)$. If $\Gamma$ is a uniform lattice in $\SU(d,1)$ and $\rho: \Gamma \to \SU(d,1)$ a homomorphism, we define
\[\tau_k(\rho) = \int_{\Gamma \backslash \H^d_\C} f^*\omega^k \wedge \omega^{d-k}~,\]
where $f:\H^d_\C \to \H^d_C$ is any smooth $\rho$-equivariant map. The number $\tau_1(\rho)$ is often called the \emph{Toledo invariant} of $\rho$, while $\tau_d(\rho)$ is the \emph{volume} of the representation $\rho$.

\begin{theo}{\ \\ \vspace{-0.5cm}} \label{t:VolQuotientsSU(d,1)}
\begin{itemize}
\item If $\Gamma$ is a uniform lattice in $\SO_0(d,1)$ and $\rho:\Gamma \to \SO_0(d,1)$ a uniformly contracting representation, then
\[\Vol \left(\Gamma_\rho \backslash \SO_0(d,1)\right) = \Vol(\SO(d)) \left | \Vol(\Gamma \backslash \H^d) + (-1)^d \Vol(\rho)\right|~.\]
\item If $\Gamma$ is a uniform lattice in $\SU(d,1)$ and $\rho:\Gamma \to \SU(d,1)$ is a uniformly contracting representation, then
\[\Vol \left(\Gamma_\rho \backslash \SU(d,1)\right) = \Vol(\SU(d+1)) \left|\sum_{k=0}^d \tau_k(\rho)\right|~.\]
\end{itemize}
\end{theo}

\begin{proof}
The compact symmetric space dual to $\H^d$ is $\S^d$, whose cohomology ring is generated by $\1$ and the fundamental class. We deduce that the Lefschetz form of $\H^d \times \H^d$ is
\[\frac{1}{\Vol(\S^d)} \left(\vol_{\H^d} \otimes \1 + (-1)^d \1  \otimes\vol_{\H^d} \right)~.\]

Clearly, Theorem \ref{t:VolQuotientsSU(d,1)} is consistent with taking finite index subgroups. By Selberg's lemma, we can thus assume that $\Gamma$ is torsion-free. Let $f: \H^d \to \H^d$ be a smooth $\rho$-equivariant map. Then the graph of $f$ is a $\Gamma_\rho$-invariant submanifold of dimension $d$ of $\H^d \times \H^d$ on which $\Gamma_\rho$ acts freely, properly discontinuously and cocompactly. Let us denote by $\Graph(f)$ its quotient by $\Gamma_\rho$:
\[\mathrm{Graph}(f) = \Gamma_\rho \backslash \{(x,f(x)), x\in \H^d\} \subset \Gamma_\rho \backslash \H^d\times \H^d~.\]
Then $\Graph(f)$ represents the homology class $[\Gamma_\rho]$ and by Theorem \ref{t:VolumeCliffordKlein}, we have
\begin{eqnarray*}
\Vol \left(\Gamma_\rho \backslash \SO_0(d,1)\right) & = & \frac{\Vol(\SO(d+1))}{\Vol(\S^d)} \left |\int_{\Graph(f)} \vol_{\H^d} \otimes \1 + (-1)^d \1  \otimes\vol_{\H^d} \right| \\
& = & \Vol(\SO(d)) \left | \int_{\Gamma \backslash \H^d} \vol_{\H^d} \wedge f^* \1 + (-1)^d \1 \wedge f^*\vol_{\H^d} \right | \\
& = & \Vol(\SO(d)) \left | \Vol(\Gamma \backslash \H^d) + (-1)^d \Vol(\rho) \right |~.
\end{eqnarray*}

Similarly, the integral cohomology ring of $\ProjC{d}$ is generated by the powers of the form symplectic form $\omega^U$. We deduce that the Lefschetz form of $\H^d_\C \times \H^d_\C$ is
\[\sum_{k=0}^d \omega^k \otimes \omega^{d-k}~.\]

Let $f:\H^d_\C \to \H^d_\C$ be a smooth $\rho$-equivariant map and define
\[\mathrm{Graph}(f) = \Gamma_\rho \backslash \{(x,f(x)), x\in \H^d_\C\} \subset \Gamma_\rho \backslash \H^d_\C\times \H^d_\C~.\]

As in the $\SO_0(d,1)$ case, we have
\begin{eqnarray*}
\Vol \left(\Gamma_\rho \backslash \SU(d,1)\right) & = & \Vol(\SU(d+1)) \left |\int_{\Graph(f)} \sum_{k=0}^d \omega^k \otimes \omega^{d-k} \right| \\
& = & \Vol(\SU(d+1)) \left | \sum_{k=0}^d \int_{\Gamma \backslash \H^d_\C} \omega^k \wedge f^* \omega^{d-k} \right | \\
& = & \Vol(\SU(d+1)) \left | \sum_{k=0}^d \tau_k(\rho) \right |~.
\end{eqnarray*}
  
\end{proof}

\section{Obstruction to the existence of compact Clifford--Klein forms} \label{s:NonExistence}

In this section, we return to the general case of a reductive homogeneous space $G/H$.

Assume that the form $\omega_{G,H}$ (or equivalently, the form $\omega_{G,H}^U$) vanishes. Then Theorem \ref{t:VolumeCliffordKlein} implies that the volume of a compact quotient of $G/H$ should be~$0$. Therefore, such a compact quotient simply cannot exist.

As a first application of this obstruction, one obtains a proof of Kobayashi's rank conjecture (Theorem \ref{t:KobayashiRankConj}), which follows directly from the first point of Theorem \ref{t:RankCondition}:

\begin{theo}
If $\rk(G) - \rk(K) < \rk(H) - \rk(L)$, then $G/H$ does not have compact quotients.
\end{theo}

Unfortunately, this theorem does not provide any new example of homogeneous spaces without compact quotients. Indeed, Morita independently proved in \cite{MoritaPreprint} that this theorem is implied by the cohomological obstruction he described in \cite{Morita15}.\\

In this section, we give three other ways of proving that the form $\omega_{G,H}$ vanishes, leading to the proof of Theorem \ref{t:AdvanceKobayashiConj}.

\begin{theo} \label{t:VanishingForm1}
For the following pairs $(G,H)$, the volume form $\omega_{G,H}$ vanishes and $G/H$ does not admit any compact Clifford--Klein form.
\begin{itemize}
\item[$(1)$] $G= \SO_0(p,q+r)$, $H= \SO_0(p,q)$, $p,q,r>0$, $p$ odd;
\item[$(2)$] $G= \SL(n,\R)$, $H= \SL(m,\R)$, $1<m<n$, $m$ even.
\end{itemize}
\end{theo}

\begin{proof}[Proof of Theorem \ref{t:VanishingForm1}]

Recall that, by Lemma \ref{l:ComputationOmegaGH}, the form $\omega_{G,H}$ at the point $x_0 = K$ is given by
\[(\omega_{G,H})_{x_0} = \int_{K/L} \Ad_u^* \omega_{\k^\perp \cap \h^\perp}\ \d \omega_{K/L}(u)~.\]

In both cases, we exhibit an element $\Omega \in K$ whose action on $\g$ stabilizes $\k^\perp \cap \h^\perp$ and whose induced action on $\k^\perp \cap \h^\perp$ has determinant~$-1$. It follows that
\begin{eqnarray*}
\omega_{G,H} & = & \int_{K/L} {\Ad_U}_* \omega_{\k^\perp\cap \h^\perp}\ \d \vol_{K/L}(U) \\
\ & = & \int_{K/L} {\Ad_{U\Omega}}_* \omega_{\k^\perp\cap \h^\perp}\ \d \vol_{K/L}(U)\\
\ & = & \int_{K/L} - {\Ad_U}_* \omega_{\k^\perp\cap \h^\perp}\ \d \vol_{K/L}(U)\\
\ & = & - \omega_{G,H}~,
\end{eqnarray*}
hence $\omega_{G,H} = 0$. 

For both cases in Theorem \ref{t:VanishingForm1}, we now describe $\k^\perp \cap \h^\perp$ as a space of matrices and we give a choice of an element $\Omega$. This element $\Omega$ simply multiplies certain coefficients of the matrices in $\k^\perp\cap \h^\perp$ by $-1$ and we leave to the reader the verification that the induced action on $\k^\perp\cap \h^\perp$ has determinant $-1$.

\begin{itemize}
\item[(1)] $G= \SO_0(p,q+r)$, $H= \SO_0(p,q)$, $p,q,r>0$, $p$ odd:\\

In this case, $K=\SO(p) \times \SO(q+r)$ and $\k^\perp\cap \h^\perp$ is the space of matrices of the form
\[
\left( \begin{array}{@{}C{2cm}@{}|@{}C{1.5cm}@{}}
 \Huge{0} & \begin{array}{@{}C{0pt}@{}C{0.7cm}@{}|@{}C{0.8cm}@{}} \rule{0pt}{2cm} & 0 & \transp{A} \\[-4pt] \end{array} \\ \hline
\begin{array}{@{}C{0pt}@{}C{2cm}@{}} \rule{0pt}{0.7cm} & 0 \\[-4pt] \hline \rule{0pt}{0.8cm} & A \\[-4pt] \end{array} & \Huge{0} \\
\end{array} \right)~,\]
with $A \in \M_{r,p}(\R)$. We take $\Omega$ to be the diagonal matrix such that $\Omega_{ii} = -1$ when $i = p+q$ or $p+q+1$ and $\Omega_{ii} = 1$ otherwise.\\

\item[(2)] $G= \SL(n,\R)$, $H= \SL(m,\R)$, $m$ even:\\

In this case, $K=\SO(n)$ and $\k^\perp\cap \h^\perp$ is the space of matrices of the form
\[
\left( \begin{array}{@{}C{0pt}@{}C{2cm}@{}|@{}C{1.5cm}@{}}
 \rule{0pt}{2cm} & \lambda \I_{m} & A \\[-4pt] \hline
 \rule{0pt}{1.5cm} & \transp{A} & B \\[-4pt]
 \end{array}
 \right) ~,\]
with $A \in \M_{m,n-m}(\R)$, $B \in \Sym_{n-m}(\R)$ and $\lambda \in \R$ satisfying $\Tr(B)+ m \lambda = 0$. We take $\Omega$ to be the diagonal matrix such that $\Omega_{ii} = -1$ when $i = m$ or $m+1$ and $\Omega_{ii} = 1$ otherwise.

\end{itemize}

\end{proof}

We now turn to another way of proving that $\omega_{G,H}$ vanishes. Recall that $\omega_{G,H}$ vanishes if an only if the corresponding form $\omega_{G,H}^U$ on $G_U/K$ vanishes. By Theorem \ref{t:PoincareDual}, this happens whenever $\iota_*[H_U/L]$ vanishes in $\HH_\bullet(G_U/K, \Q)$.

\begin{theo} \label{t:VanishingForm2}
If $G$ is the complexification of $H$, then the form $\omega_{G,H}$ vanishes if and only if $\HH^\bullet_{even}(H_U/L,\Q) \neq 0$. In particular, for the following pairs $(G,H)$, the space $G/H$ has no compact Clifford--Klein form:
\begin{itemize}
\item[$(3)$] $G= \SO(p+q,\C)$, $H = \SO_0(p,q)$, $p, q>1$ or $p=1$ and $q$ even;
\item[$(4)$] $G = \SL(p+q,\C)$, $H = \SU(p,q)$, $p,q>0$;
\item[$(5)$] $G = \Sp(2(p+q),\C)$, $H = \Sp(p,q)$;
\item[$(6)$] $G = \SO(2n,\C)$, $H = \SO^*(2n)$.
\end{itemize}

\end{theo}

\begin{proof}
Since $G$ is the complexification of $H$, we have $H_U = K$. Since $G$ is a complex Lie group, we have $G_U = K\times K$. It follows that $G_U/K$ is the group space $K$ and that $H_U/L = K/L$ is mapped to $K$ by
\[\iota: g \mapsto g\ \theta(g)^{-1}~,\]
where $\theta$ is the involution of $H_U$ whose fixed point set is $L$. By Proposition \ref{t:PoincareDual}, $\omega_{G,H}$ does not vanish if and only if $\iota_*[H_U/L]$ does not vanish in $\HH_\bullet(H_U,\Q)$, which happens if and only if the image of $\iota^*$ contains a non-zero cohomology class of degree $\dim(H_U/L)$.

By the work of Cartan \cite{Cartan50}, the cohomology algebra of $H_U$ is generated by bi-invariant forms of odd degree. Moreover, $\iota^*$ maps $\HH^\bullet(H_U,\Q)$ surjectively to $\HH^\bullet_{odd}(H_U/L,\Q)$. Since $\HH^\bullet(H_U/L,\Q) = \HH^\bullet_{odd}(H_U/L,\Q)\otimes \HH^\bullet_{even}(H_U/L,\Q)$, the image of $\iota^*$ contains a form of top degree if and only if $\HH^\bullet_{even}(H_U/L,\Q) \equiv 0$.

Let us now prove $(3)$, $(4)$, $(5)$ and $(6)$. For $H= \SU(p,q)$, $\Sp(p,q)$, $\SO^*(2n)$ or $\SO_0(p,q)$ with $p$ or $q$ even, one actually has $\rk(H_U) = \rk(L)$. Therefore the cohomology of $H_U/L$ is concentrated in even degree and the image of the map $\iota^*$ is trivial. In particular, it does not contain a non-zero class of top degree.

It remains to treat the case where $H = \SO_0(p,q)$ with $p$ and $q$ odd. Note that $H_U/L$ is the Grassmannian of $p$-planes in $\R^{p+q}$. In that case, $\rk(H_U) - \rk(L) = 1$ and $\HH^\bullet_{odd}(H_U/L)$ thus has dimension $1$. If $\HH^\bullet_{even}(H_U/L)$ vanished, then the whole cohomology algebra of $H_U/L$ would be one dimensional. This is well-known to be true if and only if $p$ or $q$ equals $1$.
\end{proof}

\begin{theo} \label{t:VanishingForm3}
For the following pairs $(G,H)$, the volume form $\omega_{G,H}$ vanishes and $G/H$ does not admit any compact Clifford--Klein form.
\begin{itemize}
\item[$(7)$] $G = \SL(p+q,\R)$, $H = \SO_0(p,q)$, $p,q>1$;
\item[$(8)$] $G = \SL(p+q, \mathcal{H})$, $H = \Sp(p,q)$, $p,q>1$.
\end{itemize}
(Here $\mathcal{H}$ denotes de field of quaternions.)
\end{theo}

\begin{proof}
Again, we prove that $\iota_*[H_U/L]$ vanishes in $\HH_\bullet(G_U/K)$, this time by showing that $H_U/L$ is homotopically trivial in $G_U/K$.\\

The compact dual to $\SL(p+q,\R)$ is $\SU(p+q)$. Let us set $V = \R^p\times \{0\}$ and $W= \{0 \} \times \R^q$ in $\C^{p+q}$. Then we can identify $K$ with $\Stab(V\oplus W) \subset \SU(p+q)$, $H_U$ with $\Stab(V \oplus i W)$, and $L$ with $\Stab(V) \cap \Stab(W)$.

For $t\in [0,1]$, let $g_t$ be the map in $\U(p+q)$ defined by
\[g_t(x) = x \textrm{ if } x\in V~,\]
\[g_t(x) = e^{\frac{it\pi}{2}} x \textrm{ if } x\in W~.\]
The conjugation by $g_t$ preserves $L$ and one can thus define
\[\function{\phi_t}{H_U/L}{G_U/K}{hL}{g_t h g_t^{-1} K~.}\]
The conjugation by $g_t$ sends $H_U=\Stab(V \oplus i W)$ to $\Stab(V \oplus i e^{\frac{it\pi}{2}} W)$. In particular, $\phi_0$ is the map $\iota: H_U/L \to G/K$, and $\phi_1$ sends $H_U/L$ to a point. Therefore the map $\iota :H_U/L \to G_U/K$ is homotopically trivial, and in particular $\iota_*[H_U/L] = 0$ in $\HH_\bullet(G_U/K)$.\\

Case $(8)$ can be treated similarly: set $V = \C^p\times \{0\}$ and $W= \{0 \} \times \C^q$ in $\mathcal{H}^{p+q}$. Then $G_U = \Sp(p+q)$ and one can identify $K$ with $\Stab(V\oplus W)$, $H_U$ with $\Stab(V \oplus j W)$ (where $i,j,k$ denote the three complex structures defining the quaternionic structure of $\mathcal{H}$), and $L$ with $\Stab(V) \cap \Stab(W)$. One obtains the same contradiction as before by conjugating $H_U$ by the linear transformation $g_t$ that is the identity on $V$ and the multiplication by $e^{\frac{t\pi}{2} j}$ on $W$.
\end{proof}

\subsection{Relation to earlier works} \label{ss:EarlierResults}

In the past decades, many different works have been devoted to finding various obstructions to the existence of compact Clifford--Klein forms. Let us detail where Theorems \ref{t:VanishingForm1}, \ref{t:VanishingForm2} and \ref{t:VanishingForm3} fit in this litterature.

\begin{itemize}
\item Case $(1)$ of Theorem \ref{t:VanishingForm1} extends results of Kulkarni \cite{Kulkarni81}, Kobayashi--Ono \cite{KobayashiOno90} and their recent improvement by Morita \cite{Morita15}, where both $p$ and $q$ are assumed to be odd. When specified to $r=1$, we obtain in particular that $\H^{p,q} = \SO_0(p,q+1)/\SO_0(p,q)$ does not admit a compact quotient when $p$ is odd. This is an important step toward Kobayashi's space form conjecture.\\

\item The case of $\SL(n,\R)/ \SL(m,\R)$ has also been extensively studied. It is conjectured that $\SL(n,\R)/ \SL(m,\R)$ never admits a compact quotient for $1 < m < n$ (see for instance \cite[Conjecture 3.3.10]{KobayashiYoshino05}). Kobayashi proved that such quotients do not exist for $n< \lceil 3/2 m \rceil$ \cite{Kobayashi92} and Labourie, Mozes and Zimmer extended the result to $m\leq n-3$ with completely different methods (\cite{Zimmer94}, \cite{LMZ95}, \cite{LabourieZimmer95}). On the other side, Benoist proved that $\SL(2n+1,\R)/\SL(2n,\R)$ does not admit a compact quotient \cite{Benoist96}. Case $(2)$ of Theorem \ref{t:VanishingForm1} recovers Benoist's result\footnote{Benoist's result is actually stronger: every discrete group acting properly discontinuously on $\SL(2n+1,\R)/\SL(2n,\R)$ is virtually Abelian.} and also implies that $\SL(2n+2,\R)/\SL(2n,\R)$ does not admit a compact quotient, which was previously known only for $n=1$ \cite{Shalom00}.\\

\item Theorem \ref{t:VanishingForm2} is mostly new. Note that the so-called \emph{Calabi--Markus phenomenon} implies that the symmetric spaces $\SL(n,\C)/\SL(n,\R)$ and $\Sp(2n,\C)/\Sp(2n,\R)$ do not admit compact Clifford--Klein forms. Therefore, the only classical Lie groups $H$ for which $H_\C/H$ might admit a compact Clifford--Klein form are $\SO(p,1)$ with $p$ even and $\SL(n,\mathcal{H})$ (where $\mathcal{H}$ denotes the quaternions). Interestingly, the homogeneous space $\SO(8,\C)/\SO(7,1)$ is known to admit compact Clifford--Klein forms (see \cite[Corollary 3.3.7]{KobayashiYoshino05}).\\

\item Theorem \ref{t:VanishingForm3} improves a recent result of Morita \cite{Morita15}, where $p$ and $q$ are assumed to be odd. It was first proved by Kobayashi when $p=q$ \cite{Kobayashi96} and by Benoist when $p=q+1$ \cite{Benoist96}. More precisely, Benoist proved that every discrete group acting properly discontinuously on $\SL(2p+1)/\SO_0(p,p+1)$ is virtually Abelian (in particular, its action is not cocompact). He also constructed proper actions of a free group of rank $2$ as soon as $p\neq q$ or $q+1$.\\

\end{itemize}

The proof of Theorem \ref{t:VanishingForm1} can be adapted to show the vanishing of $\omega_{G,H}$ in many other cases that we did not include because the non-existence of compact Clifford--Klein forms was already known. We can prove for instance that $\SL(n,\R)/\SL(m,\R) \times \SL(n-m,\R)$ does not have any compact quotient for $0<m<n$, $n$ odd (see \cite{Benoist96}), that $\SO(n,\C)/\SO(m,\C) \times \SO(n-m,\C)$ does not have any compact quotient for $1<m<n-1$, $n$ odd (see \cite{Kobayashi92}), or that $\SO(n,\C)/\SO(m,\C)$ does not have any compact quotient for $1<m<n$, $m$ even (see \cite{Kobayashi96,Benoist96}).\\

\subsection{Relation to Yosuke Morita's work} \label{ss:Morita}

The first version of this article did not contain Sections \ref{s:CompactDual}, \ref{s:InclusionSymSpaces} and \ref{s:GroupSpaces}. Section \ref{s:Rigidity} stated a theorem of \emph{local rigidity} of the volume and Section \ref{s:NonExistence} contained only a refinded version of Theorem\ref{t:VanishingForm1}. After our preprint appeared on arXiv, Yosuke Morita posted a preprint where he uses a cohomological obstruction to prove the non-existence of compact quotients of certain reductive homogeneous spaces. In particular, he obtained Theorems \ref{t:VanishingForm1}, \ref{t:VanishingForm2} and \ref{t:VanishingForm3}. This motivated me to find new ways of proving the vanishing of the form $\omega_{G,H}$ and led me to the compact duality argument and theorems \ref{t:VanishingForm2} and \ref{t:VanishingForm3} which improved significantly this paper.

After discussing with Morita, it seems likely, though not obvious, that our two obstructions are in fact equivalent. We hope to prove this equivalence in a future work.

\section{Local foliations of $G/H$ and global foliations of $\Gamma \backslash G/H$} \label{s:LocalFibrations}

The results of this paper where driven by the idea that compact Clifford--Klein forms $\Gamma \backslash G/H$ should ``look like'' $(K/L)$-bundles over a classifying space for $\Gamma$. This was suggested by the following theorem:

\begin{CiteThm}[Gu\'eritaud--Kassel, \cite{GueritaudKassel}] \label{t:FibrationGK}
Let $\Gamma$ be a discrete torsion-free subgroup of $\SO_0(d,1) \times \SO_0(d,1)$ acting properly discontinuously and cocompactly on $\SO_0(d,1)$ (by left and right multiplication). Then $\Gamma$ is isomorphic to the fundamental group of a closed hyperbolic $d$-manifold $B$, and $\Gamma \backslash \SO_0(d,1)$ admits a fibration over $B$ with fibers of the form
\[g \SO(d) h^{-1}, \quad g,h \in \SO_0(d,1)~.\]
\end{CiteThm}

More generally, we conjecture the following:
\begin{conj}
Let $G/H$ be a reductive homogeneous space (with $G$ and $H$ connected), $L$ a maximal compact subgroup of $H$ and $K$ a maximal compact subgroup of $G$ containing $L$. Let $\Gamma$ be a torsion free discrete subgroup of $G$ acting properly discontinuously and cocompactly on $G/H$. Then there exists a closed manifold $B$ of dimension $p$ such that
\begin{itemize}
\item the fundamental group of $B$ is isomorphic to $\Gamma$,
\item the universal cover of $B$ is contractible,
\item $\Gamma \backslash G/H$ admits a fibration over $B$ with fibers of the form $g K/L$ for some $g\in G$.
\end{itemize}
\end{conj}

To support this conjecture, we note that the vanishing of the form $\omega_{G,H}$ (which implies the non-existence of compact Clifford--Klein forms) is actually an obstruction to the existence of a \emph{local} fibration by copies of $K/L$.

\begin{prop} \label{p:NoLocalFoliation}
Let $G/H$ be a reductive homogeneous space (with $G$ and $H$ connected), $L$ a maximal compact subgroup of $H$ and $K$ a maximal compact subgroup of $G$ containing $L$. If the form $\omega_{G,H}$ on $G/K$ vanishes (and in particular for all the pairs $(G,H)$ in Theorem \ref{t:AdvanceKobayashiConj}), then no non-empty open domain of $G/H$ admits a foliation with leaves of the form $g K/L$. 
\end{prop}

The non-existence of such local foliations in certain homogeneous spaces may be quite surprising. For instance, if $G= \SO_0(2n-1,2)$ and $H = \SO_0(2n-1,1)$, then $G/H$ is the \emph{anti-de Sitter space} $\AdS_{2n}$ (for which the non-existence of compact Clifford--Klein forms was proven by Kulkarni \cite{Kulkarni81}). In that case, $K/L$ is a timelike geodesic and we obtain the following corollary:

\begin{coro}
No open domain of the even dimensional anti-de Sitter space can be foliated by complete timelike geodesics.
\end{coro}

This leads to the following more general question, that may be of independent interest:
\begin{ques}
Let $G/H$ be a reductive homogeneous space, $G'$ a closed subgroup of $G$ and $H'=G'\cap H$. When does $G/H$ admit an open domain with a foliation by leaves of the form $g G'/H'$?
\end{ques}

\begin{proof}[Proof of Proposition \ref{p:NoLocalFoliation}]
Assume that there exists a non-empty domain $U$ in $X=G/H$ with a foliation by leaves $(F_v)_{v\in V}$ of the form $g_v K/L$. Since the stabilizer in $G$ of $K/L \subset G/H$ is exactly $K$, the space of leaves $V$ can be seen as a submanifold of dimension $p$ in $G/K$. Set $U' = \pi^{-1}(V)$, where $\pi$ is the projection from $G/L$ to $G/K$. Then the projection $\psi$ from $G/L$ to $G/H$ induces a diffeomorphism from $U'$ to $U$. We thus have
\[\int_U \vol_X = \int_{U'} \psi^* \vol_X~.\]
On the other hand, by construction of $\omega_{G,H}$, we have
\[\int_{U'} \psi^* \vol_X = \int_V \omega_{G,H}~.\]
Since $U$ is non-empty, its volume is non-zero, hence the form $\omega_{G,H}$ cannot vanish.
\end{proof}

\bibliographystyle{plain}
\bibliography{biblio}

\end{document}